\documentclass[a4paper, 11pt]{paper}
\usepackage{geometry}

\usepackage[english]{babel}
\usepackage[utf8x]{inputenc}
\usepackage{amsmath}
\usepackage{amsthm}
\usepackage{amssymb}
\usepackage{dsfont}
\usepackage{graphicx}
\graphicspath{{pictures/}}
\usepackage[textsize=footnotesize]{todonotes}
\usepackage{listing}
\usepackage{xcolor}
\usepackage{textcomp}
\usepackage{comment}
\usepackage{wrapfig}
\usepackage{physics}
\usepackage{mathtools}
\usepackage{enumitem}
\usepackage{nameref}
\usepackage{thmtools}
\usepackage{pdfpages}

\RequirePackage[nameinlink, capitalize, noabbrev]{cleveref}
\crefname{observation}{Observation}{Observations}
\Crefname{observation}{Observation}{Observations}
\crefname{conjecture}{Conjecture}{Conjectures}
\Crefname{conjecture}{Conjecture}{Conjectures}
\crefname{notation}{Notation}{Notations}
\Crefname{notation}{Notation}{Notations}

\usepackage{tikz,pgf,pgfplots}
\pgfplotsset{compat=1.16}

\usepackage{tikzscale}
\usepackage{tikz}
\usetikzlibrary{shapes,arrows,calc,math,3d}
\usetikzlibrary{external}

\usepackage{subcaption}

\title{Coverage of the unit cube by dynamic Boolean models}
\author{Hanna Döring
\thanks{Osnabrück University, Germany, hanna.doering@uni-osnabrueck.de} \and Lianne de Jonge\thanks{Osnabrück University, Germany, lianne.de.jonge@uni-osnabrueck.de} \and Xiaochuan Yang\thanks{xiaochuan.j.yang@gmail.com}}
\date{\today}

\theoremstyle{definition}
\newtheorem{defn}{Definition}[section]
\theoremstyle{plain}
\newtheorem{prop}[defn]{Proposition}
\newtheorem{rem}[defn]{Remark}

\newtheorem{thm}[defn]{Theorem}

\newtheorem{lem}[defn]{Lemma}
\newtheorem{cond}[defn]{Condition}

\renewcommand{\l}{\left}
\renewcommand{\r}{\right}

\DeclarePairedDelimiter\ceil{\lceil}{\rceil}

\DeclareMathOperator{\sgn}{sgn}

\newcommand{\ind}{\mathds{1}}
\newcommand{\indi}[1]{\ind_{\l\{#1\r\}}}

\newcommand{\E}{\mathbb{E}}

\newcommand{\N}{\mathbb{N}}

\renewcommand{\P}{\mathbb{P}}
\newcommand{\Q}{\mathbb{Q}}
\newcommand{\R}{\mathbb{R}}
\renewcommand{\S}{\mathbb{S}}

\newcommand{\cC}{\mathcal{C}}

\newcommand{\cI}{\mathcal{I}}

\newcommand{\cO}{\mathcal{O}}

\newcommand{\cS}{\mathcal{S}}

\newcommand{\cW}{\mathcal{W}}
\newcommand{\cX}{\mathcal{X}}

\renewcommand{\epsilon}{\varepsilon}
\renewcommand{\phi}{\varphi}

\begin{document}

\maketitle
\begin{abstract}
    Motivated by peer-to-peer telecommunication, we study a dynamic Boolean model. We define a Poisson number of random lines through the $(d-1)$-dimensional base of a $d$-dimensional unit cube and dilate them to define cylinders. Letting $\rho$ be the expected number of cylinders, the random variable of interest is the coverage radius $R_\rho$, which is the cylinder radius required to cover the $d$-dimensional unit cube. We show that $R_\rho^{d-1}$ is of the order $\log \rho / \rho$ with high probability as $\rho$ tends to infinity.
    
    We also consider alternative dynamics resulting in generalized cylinders that are generated by dilating the trajectories of stochastic processes, in particular Brownian motions. This leads to a coverage radius of the same order.
\end{abstract}


\section{Introduction}
\label{sec:introduction}
Random coverage problems have been studied extensively in many different settings. In this type of problem, there is a (possibly random) number of random sets $A_i$ and the question is whether some fixed set $A$ is contained in the union of the $A_i$'s. In an early paper \cite{stevens_solution_1939} from 1937, Stevens studied coverage of a circle by randomly placed arcs $A_i$ of length $x$. Since then, there have been many people studying some variation of the problem, see for example \cite{dvoretzky1956} and \cite{janson_random_1986}.

We study a variation on the coverage problem from \cite{penrose2021}, where Penrose defined the \emph{coverage radius} for the Boolean model to be the radius of balls in the Boolean model needed to cover a set. However, instead of covering a set by balls, we are interested in covering the unit cube $[0,1]^d$ by cylinders.

The study of cylinder models is motivated by peer-to-peer telecommunication, where it can be considered an extension of the Boolean model. 
In this model, the positions of people are represented by Poisson distributed points in $\R^2$. They carry phones that have a communication radius represented by a disk around each point and messages can be transmitted between phones when these disks overlap. This can be modeled by the Boolean model or the Gilbert graph as introduced in \cite{gilbert1961}.
Since people move around, we assign a velocity vector to each point and let it move in a single direction. Adding time as a third dimension, the disks' movements are represented by cylinders. The asymptotic properties of the volume covered by such cylinders were studied by \cite{aschenbruck2021}.

Based on this motivation, we consider a cylinder model in $\R^d$ consisting only of cylinders crossing the base of the unit cube, $[0,1]^{d-1}\times\{0\}$. The cylinders are formed by constructing a Poisson number of lines in random directions crossing $[0,1]^{d-1}\times\{0\}$. For some bounded set $K$ and a scaling parameter $r$, the Minkowski sum with $rK$ thickens the line into cylinders. See Figure \ref{fig:cylinders} for a picture of this situation in two dimensions with the unit ball $K$.
\begin{figure}
    \centering
    \includegraphics[width = .4\textwidth]{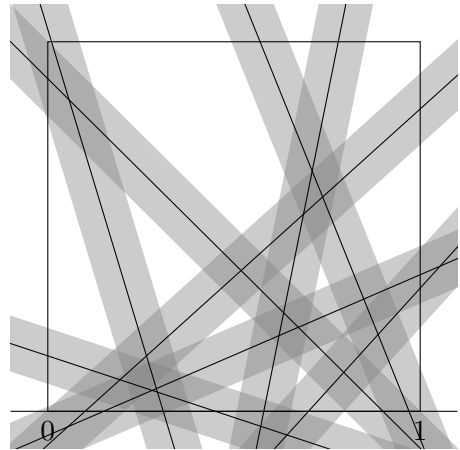}
    \caption{Cylinder model in $\R^2$ with $K=B_2(o,r)$. The lines crossing $[0,1]\times\{0\}$ are dilated to create the cylinders.}
    \label{fig:cylinders}
\end{figure}

If $r$ is large enough, the cylinders cover the unit cube, provided that at least one cylinder exists. We are interested in the asymptotics of the coverage radius: letting the intensity of $\rho$ of the number of cylinders tend to infinity, how wide should the cylinders be to cover $[0,1]^d$? In the context of telecommunication described above, we ask how large the communication radius has to be so that people starting in $[0,1]^{d-1}$ at time zero cover $[0,1]^{d-1}$ until time one.

We should note that the cylinders in this model are  different from the homogeneous Poisson cylinders that are formed by thickening the Poisson line sets introduced by \cite{goudsmit1945}. In a two-dimensional setting, the coverage radius of homogeneous cylinder sets has been studied by \cite{chenavier2016}. A higher-dimensional coverage result with a time-component can be found in \cite{broman2019random}.

For the application to telecommunication, it is more realistic to consider the possibility that people can change the direction in which they walk. In \cite{vandenberg1997} and \cite{peres_mobile_2013}, the movements are modelled by a Brownian motion. Based on this, we consider an alternative model where the cylinders are generated using a Minkowski sum with a Brownian motion instead of a line. We refer to these models where the lines are replaced by other stochastic processes as \emph{generalized cylinder models}.

In what follows, we first formally define the (generalized) cylinder model and state the results. In \cref{sec:coverage uniform}, we prove a law of large numbers for the coverage radius for uniformly distributed cylinder directions and $K$ equal to the unit ball. In \cref{sec:generalization}, we prove results for the cylinders constructed using lines or Brownian motions. The results for the cylinders are generalized to different directional distributions and sets $K$. For the Brownian motion, we consider $K$ to be the $(d-1)$-dimensional unit ball.

\subsection{Definitions and results}
\label{sec:definitions}
Let $\eta_\rho$ be a homogeneous Poisson point process on $\R^{d-1}$ with intensity measure $\rho \lambda_{d-1}$, where $\lambda_{d-1}$ denotes the $(d-1)$-dimensional Lebesgue measure and $\rho>0$. The point measure $\eta_\rho$ can be written in terms of its support as follows:
\begin{equation*}
    \eta_\rho = \sum_{i=1}^\infty \delta_{Y^{(i,\rho)}},
\end{equation*}
where $Y^{(i,\rho)}\in \R^{d-1}$ for all $i\in\N$. By a standard abuse of notation, we identify a point process with its support and write $\eta_\rho=\{Y^{(1,\rho)},Y^{(2,\rho)},\dots\}$. The points are the starting point of independent and continuous-time stochastic processes
\begin{equation*}
    \{X_t^{(i,\rho)}\colon t\in\R_+,\, X_0^{(i,\rho)}=Y^{(i,\rho)}\times \{0\}\}\subset \R^{d-1}
\end{equation*}
that represent the movement of the points through time. We require that the stochastic processes $\{X_t^{(i,\rho)}-X_0^{(i,\rho)}\}_{t\in\R_+}$ are identically distributed for all $i\in\N$, so that for all $t'\in\R_+$ the set $\{X_{t'}^{(1,\rho)}, X_{t'}^{(2,\rho)}, X_{t'}^{(3,\rho)},\dots\}$ forms the support of a homogeneous Poisson point process. We denote the trajectory of $X_t^{(i,\rho)}$ by
\begin{equation*}
    X^{(i,\rho)} \coloneqq \{(X_t^{(i,\rho)},t) \colon t\in \R_+\}\subset \R^d.
\end{equation*}

We are specifically interested in the trajectories through $[0,1]^{d-1}\times\{0\}$. Let $\xi_\rho=\eta_\rho([0,1]^{d-1})$ be the Poisson distributed number of points in $[0,1]^{d-1}$ and assume that the points in $\eta_\rho$ are indexed in such a way that $X_0^{(1,\rho)},X_0^{(2,\rho)},\dots,X_0^{(\xi_\rho,\rho)}\in [0,1]^{d-1}$. We then define
\begin{equation*}
    \cX_\rho(K) \coloneqq \bigcup_{i=1}^{\xi_\rho} X^{(i,\rho)}\oplus K.
\end{equation*}
Since this paper is motivated by cylinder sets, we refer to this set as the generalized cylinder set.

The set $K$ can be scaled by a positive parameter $r$, in which case $r K = \{r x \colon x\in K\}$. We are interested in how we need to scale $K$ to cover the entire unit cube by the generalized cylinder set. We therefore define the coverage radius by
\begin{equation*}
    R_\rho (K) = \inf\{r\in\R_+\colon \cX_\rho (rK)\supset [0,1]^d\}
\end{equation*}
analogous to the coverage threshold for the Boolean model as defined in \cite{penrose2021}.

In this paper, we are mainly interested in two types of generalized cylinder sets: cylinders generated by lines and ones where the stochastic processes are Brownian motions. Below, we define these two sets and state the results.

\paragraph{Cylinder sets}
Cylinders in $\R^d$ are generated by dilating a line. Let $S^{(i,\rho)}$, $i\in\N$, be independent, identically distributed random variables taking values on the unit hemisphere $\S^{d-1}_+ = \{(s_1,\dots,s_d)\in\R^d\colon s_1^2+\dots+s_d^2 = 1,\, s_d\geq 0\}$. We denote the probability measure of $S^{(i,\rho)}$ by $\Q$. Then, by choosing
\begin{equation}
    X_t^{(i,\rho)} = X_0^{(i,\rho)} + t S^{(i,\rho)}
    \label{eq:line process}
\end{equation}
each stochastic process defines a random line. Depending on the choice of $K$, these lines can be dilated to form cylinders of different shapes.

Let $B_d(o,r)$ be a $d$-dimensional ball with radius $r$ and $B'_{d-1}(o,r)\coloneqq B_{d-1}(o,r)\times \{0\}\subset \R^{d-1}\times \{0\}$ a $(d-1)$-dimensional ball embedded in $\R^d$. Then we define
\begin{align*}
    \cC_{\rho}(r) &\coloneqq \cX_\rho(B_d(o,r)) , \qquad \text{and}\\
    \cC'_{\rho}(r) &\coloneqq \cX_\rho(B'_{d-1}(o,r))
\end{align*}
to be the cylinder sets based on the choice of $X_t^{(i,\rho)}$ in \eqref{eq:line process}. The set $\cC_{\rho}(r)$ contains every point within Euclidean distance $r$ of the lines. The cylinders in $\cC'_{\rho}(r)$ correspond to the ones in the telecommunication model.

Let $\sigma_{d-1}$ be the normalised spherical measure, i.e. $\sigma_{d-1}(\S^{d-1}) = 1$. When the directional distribution is uniform on $\S_+^{d-1}$, the probability measure is $\Q = 2\sigma_{d-1}$. The $(d-1)$-dimensional volume of $\S^{d-1}_+$ embedded in $\R^d$ equals $\frac{d}{2}\kappa_d$, where $\kappa_d$ is the $d$-dimensional volume of the unit ball. The volume of some measurable $A\subset\S^{d-1}_+$ is therefore $d\kappa_d\sigma_{d-1}(A)$.

For $K = B_d(o,1)$ and uniformly distributed cylinder directions, we show the following convergence for the coverage radius:
\begin{thm}[Coverage using uniform angle distribution] 
    Let $\Q=2\sigma_{d-1}$ be the uniform probability measure on $\S^{d-1}_+$ and $K = B_d(o,1)$. Define
    \begin{equation*}
        \phi_d(x) = \int_{[0,1]^{d-1}}\frac{1}{\|y\times\{0\} - x\|^{d-1}} \dd y
    \end{equation*}
    for $x\in [0,1]^{d-1} \times (0,1]$. Then,
    \begin{equation*}
        \frac{R_\rho(B_d(o,1))}{\l(\log{\rho}/\rho \r)^{1/(d-1)}} \overset{\P}{\longrightarrow} \l(\frac{d^2}{d-1}\cdot \frac{\kappa_d}{2\kappa_{d-1}} \cdot \frac{1}{\inf_{x\in[0,1]^d} \phi_d(x)}\r)^{1/(d-1)} \eqqcolon (c^*)^{1/(d-1)}
    \end{equation*}
    as $\rho$ tends to infinity.
    \label{thm:coverage uniform}
\end{thm}
The function $\phi_d(x)$ follows from the probability that $x$ is covered by a cylinder. It takes its lowest values in the corners of $[0,1]^{d-1}\times\{1\}$. The infimum over all values of $\phi_d$ implies that the constant is dependent on the part of $[0,1]^d$ that is most difficult to cover.

For the extension of this result to other sets $K$ and directional distributions, it is 
natural to introduce a condition on the angle distribution such that the entire cube can be reached by the lines. Let $(s_1,s_2,\dots,s_d) \eqqcolon s\in\S^{d-1}_+$ be a point on the unit hemisphere.
Let $\sgn\colon \R\to\{-1,0,1\}$ denote the sign function. For $(z_1,\dots,z_{d-1})\in \{-1,1\}^{d-1}$, we then define
\begin{equation*}
    \cS(z_1,\dots,z_{d-1}) \coloneqq \{(s_1,\dots,s_d)\in\S_+^{d-1} \colon\, \sgn(s_1)=z_1,\,\dots,\, \sgn(s_{d-1})=z_{d-1},\, s_d\geq 2/\sqrt{5}\}.
    \label{eq:direction set}
\end{equation*}
Denoting the closure of a set $A$ by $\overline{A}$, we formulate the following condition:
\begin{cond}
    The directional random variables $S_i$ are independently and identically distributed such that
    \begin{equation*}
        \P\l(S_i\in \overline{\cS\l(z_1,\dots,z_{d-1}\r)}\r) > p_S
    \end{equation*}
    for all $(z_1,\dots,z_{d-1})\in \{-1,1\}^{d-1}$ and some constant $p_S>0$.
    \label{cond:distribution S 2}
\end{cond}
Loosely speaking, this condition says that cylinders should have a positive probability of being vertical enough in every direction.
Note that this condition is for example satisfied for uniformly distributed directions.

For more general compact sets $K$ and distributions $\Q$, the convergence to some interval is shown.
\begin{thm}[Coverage of unit cube]
    Let $K\subset\R^d$ be a bounded set containing $B'_{d-1}(o,r)$ for some $r>0$ and let $d\geq 2$. When Condition \ref{cond:distribution S 2} is satisfied, there exist constants $c_1$ and $c_2$ depending on $K$ such that
    \begin{equation*}
        \P\Big(c_1\leq \frac{R_\rho (K)}{\l(\log{\rho}/\rho \r)^{1/(d-1)}} \leq c_2\Big)\to 1
    \end{equation*}
    as $\rho$ tends to infinity.
    \label{thm:coverage}
\end{thm}

\paragraph{Brownian motion}
Let $W_t^{(i,\rho)}$ for $i\in \N$ be independent Brownian motions in $\R^{d-1}$ and take
\begin{equation*}
    X_t^{(i,\rho)} = X_0^{(i,\rho)} + W_t^{(i,\rho)}.
\end{equation*}
which tracks the movement of the Brownian motion through time. For this choice of $X_t^{(i,\rho)}$, we concentrate on the telecommunication application and consider $\cX_\rho(B'_{d-1}(o,1))$. 
We then prove the following result for the coverage radius.
\begin{thm}
    Let $X_t^{(i,\rho)} = X_0^{(i,\rho)} + W_t^{(i,\rho)}$. Then there exist constants $c_1$ and $c_2$ such that
    \begin{equation*}
        \P\l(c_1\leq \frac{R_\rho(B'_{d-1}(o,1))}{\l(\log{\rho}/\rho \r)^{1/(d-1)}} \leq c_2\r)\to 1
    \end{equation*}
    as $\rho$ tends to infinity.
    \label{thm:brownian coverage}
\end{thm}

\section{Coverage with uniform directional distribution}
\label{sec:coverage uniform}
Throughout this section, we assume that $K = B_d(o,1)$ and $\Q = 2\sigma_{d-1}$. We prepare for the proof of \cref{thm:coverage uniform} with three lemmas.

If a cylinder of radius $r$ covers a point $x$, the line corresponding to the cylinder crosses a ball of radius $r$ around the cylinder. The following lemma provides the probability that a line through $y\in[0,1]^{d-1}\times\{0\}$ crosses this ball and the cylinder therefore covers $x$.
\begin{lem}[Coverage probability]
    Let $y\in[0,1]^{d-1}$ be a point on the cube's base and $x\in [0,1]^{d}$ a point in the unit cube. Consider $r_\rho>0$ with $r_\rho\to 0$ as $\rho\to\infty$. Let $L_y=\{y\times\{0\}+\alpha S \colon \alpha\in\R\}$ be a random line through $y$ with uniformly distributed $S\in\S^{d-1}$. Then,
    \begin{equation*}
        \P\l(L_y \cap B_d(x,r_\rho) \neq \emptyset\r) =
        \begin{cases}
            \l(\frac{r_\rho}{\|y\times\{0\} -x\|}\r)^{d-1} \frac{2\kappa_{d-1}}{d\kappa_d}
            + \cO(r_\rho^{d+1})  &\text{if } \|y\times\{0\} -x\| > r_\rho
            \\
            1 &\text{otherwise}.
        \end{cases}
    \end{equation*}
    as $\rho\to\infty$, where $\kappa_{n}$ denotes the volume of an $n$-dimensional unit ball. The $\cO(r_\rho^{d+1})$-term is non-negative.
    \label{lem:line crossing probability}
\end{lem}
\begin{proof}
    If $\|y\times\{0\} -x\| \leq r_\rho$, then the probability of covering $x$ is trivially one. We therefore only need to prove the first case.
    
    The direction of the line $L_y$ is determined by a uniformly chosen point on $\S^{d-1}_+$. To determine the probability that $L_y$ goes through the ball $B_d(x,r_\rho)$, we need to derive the $(d-1)$-dimensional volume of
    \begin{equation*}
        \{s\in \S^{d-1}_+\colon \{y\times\{0\}+\alpha s\colon \alpha\in\R\}\cap B_d(x,r_\rho) \neq \emptyset\}
    \end{equation*}
    and divide this by the surface area of the $(d-1)$-dimensional unit hemisphere.
    
    We first consider the two-dimensional case, where $y\in[0,1]$ and $x\in[0,1]^2$. Using basic geometric arguments and the Taylor expansion of $\sin^{-1}(x)$ around $x=0$, we can derive
    \begin{align*}
        \P\l(L_y \cap B_d(x,r_\rho) \neq \emptyset\r) &= \frac{2}{\pi}\sin^{-1}\l(\frac{r_\rho}{\|y\times\{0\} -x\|}\r)\\
        &= \frac{2}{\pi}\frac{r_\rho}{\|y\times\{0\} -x\|} + \cO\l(\l(\frac{r_\rho}{\|y\times\{0\} -x\|}\r)^{3}\r).
    \end{align*}

    In higher dimensions, we can use the two-dimensional result to approximate the surface of a spherical cap:
    \begin{multline*}
        \l(\frac{r_\rho}{\|y\times\{0\} -x\|}\r)^{d-1} \kappa_{d-1} \leq d\kappa_d\sigma_{d-1}\l(\{s\in \S^{d-1}_+\colon \{y+\alpha s\colon \alpha\in\R\}\cap B_d(x,r_\rho) \neq \emptyset\}\r) \\
        \leq \l(\sin^{-1}\l(\frac{r_\rho}{\|y\times\{0\} -x\|}\r)\r)^{d-1} \kappa_{d-1}
    \end{multline*}
    so by the Taylor expansion of $\sin^{-1}(x)$,
    \begin{multline*}
        d\kappa_d\sigma_{d-1}\l(\{s\in \S^{d-1}_+\colon \{y+\alpha s\colon \alpha\in\R\}\cap B_d(x,r_\rho) \neq \emptyset\}\r) \\
        = \l(\frac{r_\rho}{\|y\times\{0\} -x\|}\r)^{d-1} \kappa_{d-1} + \cO(r_\rho^{d+1}),
    \end{multline*}
    where the $\cO(r_\rho^{d+1})$-term is non-negative.
    The probability in \cref{lem:line crossing intensity} follows.
\end{proof}

Recall that by the notation introduced in \cref{sec:definitions}, the random lines are denoted by $X^{(i,\rho)}\subset\R^d$.
\begin{lem}[Intensity of cylinders covering point]
    Consider the cylinder model $\cC_{\rho}(r_\rho)$ with $r_\rho\to 0$ as $\rho\to\infty$. Let $x\in [0,1]^{d-1}\times(0,1]$ be a point in the unit square. Then
    \begin{align*}
        \E\sum_{i=1}^{\xi_\rho} \indi{X^{(i,\rho)}\cap B_d(x,r_\rho)\neq \emptyset}
        &= \rho r_\rho^{d-1}\frac{2\kappa_{d-1}}{d\kappa_d}\int_{[0,1]^{d-1}}\frac{1}{\|y\times\{0\} - x\|^{d-1}} \dd y  + \cO(\rho r_\rho^{d+1})\\
        &= \rho r_\rho^{d-1}\frac{2\kappa_{d-1}}{d\kappa_d} \phi_d(x) + \cO(\rho r_\rho^{d+1}),
    \end{align*}
    where the $\cO(\rho r_\rho^{d+1})$-term is non-negative.
    The number of cylinders covering $x$ follows a Poisson distribution with mean given by this expectation.
    \label{lem:line crossing intensity}
\end{lem}
\begin{proof}
    This formula follows from an application of the Mecke formula in combination with \cref{lem:line crossing probability}.

    The starting points in $[0,1]^{d-1}$ and the directions of the lines form a marked Poisson point process on $[0,1]^{d-1}\times \S^{d-1}_+$. Consider
    \begin{equation*}
        A = \{(x,s)\in [0,1]^{d-1}\times \S^{d-1}_+ \colon \{x\times\{0\}+\alpha s\colon s\in\R+\}\cap B_d(x,r_\rho)\neq \emptyset\}.
    \end{equation*}
    Then $\eta_\rho(A) = \sum_{i=1}^{\xi_\rho} \indi{X^{(i,\rho)}\cap B_d(x,r_\rho)\neq \emptyset}$ is Poisson distributed by definition of the Poisson point process.
\end{proof}

\cref{lem:line crossing intensity} can be used to calculate the expected volume of a subset of $[0,1]^{d}$ that remains uncovered.
\begin{lem}[Uncovered volume]
    Let $r_\rho= \l(c\frac{\log \rho}{\rho}\r)^{1/(d-1)}$ for some constant $c>0$ be the cylinder radius at intensity $\rho>0$ and $D\subset [0,1]^{d-1}\times (0,1]$. Then,
    \begin{equation*}
        \E \lambda_{d}\l(D\setminus \cC_{\rho}(r_\rho)\r) = \int_{D} \rho^{-c\frac{2\kappa_{d-1}}{d\kappa_d} \phi_d(x)}\l(1+\cO(r_\rho^{d+1})\r) \dd x
    \end{equation*}
    as $\rho$ tends to infinity.
    \label{lem:uncovered volume}
\end{lem}
\begin{proof}
    Combining $r= \l(c\frac{\log \rho}{\rho}\r)^{1/(d-1)}$ with \cref{lem:line crossing intensity}, the number of cylinders covering a point $x$ is Poisson distributed with intensity 
    \begin{align*}
        \E\sum_{i=1}^{\xi_\rho} \indi{X^{(i,\rho)}\cap B_d(x,r)\neq \emptyset}&=\rho r^{d-1}\frac{2\kappa_{d-1}}{d\kappa_d} \phi_d(x) + \cO(\rho r^{d+1}) \\
        &= c \log{\rho} \frac{2\kappa_{d-1}}{d\kappa_d} \phi_d(x) + \cO(\rho r^{d+1}).
    \end{align*}
    Then, using the probability that a Poisson-distributed random variable is zero,
    \begin{equation*}
        \P\l(x\notin \cC_{\rho}(r_\rho)\r) = \exp\l(-c \log{\rho} \frac{2\kappa_{d-1}}{d\kappa_d} \phi_d(x)\r)\l(1+ \cO(\rho r^{d+1})\r),
    \end{equation*}
    where the factor $1+ \cO(\rho r^{d+1})$ follows from the Taylor series of $\exp(x)$.
    The expectation can be calculated using the Mecke formula, resulting in the integral over all points $x\in D$.
\end{proof}
We can now use these lemmas to prove \cref{thm:coverage uniform}.

\begin{proof}[Proof of \cref{thm:coverage uniform}]
    For the duration of this proof, let $r$ be of the form
    \begin{equation*}
        r = r_\rho =  \l(c \frac{\log \rho}{\rho}\r)^{1/{(d-1)}}
    \end{equation*}
    for some constant $c>0$, where $r_\rho$ is used when we want to emphasize the dependency on $\rho$. Let
    \begin{equation*}
        c^*\coloneqq \frac{d^2 \kappa_d}{2(d-1)\kappa_{d-1}} \cdot \frac{1}{\inf_{x\in[0,1]^d} \phi_d(x)}
    \end{equation*}
    be the critical value of $c$ we want to derive.    
    To show the convergence in \cref{thm:coverage uniform}, one needs to prove the following:
    \begin{equation*}
        \lim_{\rho\to\infty} \P\l([0,1]^d \subset \cC_{\rho}(r_\rho)\r) =
        \begin{cases}
            0 &\text{if } c < c^*,\\
            1 &\text{if } c > c^*.
        \end{cases}
    \end{equation*}
    
    \paragraph{Lower bound} The lower bound on the constant $c^*$ follows from an application of the second moment method to the non-covered volume. Let $\epsilon>0$ be some arbitrarily small constant and define
    \begin{equation*}
        D_\epsilon = \l\{x\in[0,1]^d\colon \phi_d(x) < \inf_{x\in[0,1]^d} \phi_d(x) + \epsilon\r\}
    \end{equation*}
    and
    \begin{equation*}
        c = \frac{d^2 \kappa_d}{2(d-1)\kappa_{d-1}} \cdot \frac{1}{\inf_{x\in[0,1]^d} \phi_d(x)+2\epsilon}<c^*.
    \end{equation*}
    The set $D_\epsilon$ consists of some points near the corners of $[0,1]^{d-1}\times\{1\}$, since the function $\phi_d(x)$ takes its smallest values there.
    
    We denote the volume of the uncovered subset of $D_\epsilon$ by $V$, so
    \begin{equation*}
        V\coloneqq \lambda_d\l(D_\epsilon\setminus \cC_{\rho}(r_\rho)\r).
    \end{equation*}
    By the second moment method,
    \begin{equation*}
        \P\l(V>0\r) \geq \frac{(\E V)^2}{\E V^2}
    \end{equation*}
    for $V\geq 0$ with finite variance.
    The expectation $\E V$ is given in \cref{lem:uncovered volume}. It remains to show that $\E V^2$ converges to $(\E V)^2$ by finding upper bounds on this expectation.

    The expectation $\E V^2$ can be written as follows:
    \begin{align*}
        \E V^2  &= \int_{D_\epsilon}\int_{D_\epsilon} \P\l(\{x,y\} \not\subset \cC_{\rho}(r_\rho)\r)\dd y \dd x\\
        &= \int_{D_\epsilon}\P\l(x\notin \cC_{\rho}(r_\rho)\r) \int_{D_\epsilon} \P\l(y\notin \cC_{\rho}(r_\rho)| x\notin \cC_{\rho}(r_\rho)\r) \dd y \dd x.
    \end{align*}
    If $x$ and $y$ are close together, their coverage is more correlated than if the distance is large. The following upper bound holds for all $\alpha\geq 0$:
    \begin{multline}
        \E V^2 \leq \int_{D_\epsilon}\int_{B_d\l(x,\rho^\alpha r\r)}\P\l(x\notin \cC_{\rho}(r_\rho)\r) \dd y \dd x \\
        + \int_{D_\epsilon}\int_{D_\epsilon \setminus B_d\l(x,\rho^\alpha r\r)}\P\l(x\notin \cC_{\rho}(r_\rho)\r)\P\l(y\notin \cC_{\rho}(r_\rho)| x\notin \cC_{\rho}(r_\rho)\r) \dd y \dd x.
        \label{eq:Y2 integral division}
    \end{multline}
    In the first integral, all $y$ near $x$ are assumed to be uncovered whenever $x$ is uncovered.
    
    The first summand can be simplified to
    \begin{align*}
        \int_{D_\epsilon} \int_{B_d\l(x,\rho^\alpha r\r)}\P\l(x\notin \cC_{\rho}(r_\rho)\r) \dd y \dd x &= \kappa_d (\rho^\alpha r)^d \int_{D_\epsilon}\P\l(x\notin \cC_{\rho}(r_\rho)\r) \dd x\\
        &= \kappa_d (c\log \rho)^{\frac{d}{d-1}} \rho^{d\alpha - \frac{d}{d-1}} \int_{D_\epsilon} \rho^{-c\frac{2\kappa_{d-1}}{d\kappa_d} \phi_d(x)}(1+\cO(\rho r_\rho^{d+1}))\dd x,
    \end{align*}
    where the integral on the right-hand side equals the expectation $\E V$.
    With $c$ defined as above,
    \begin{equation*}
        c\frac{2\kappa_{d-1}}{d\kappa_d} \phi_d(x) \leq \frac{d}{d-1} \frac{\inf_{y\in[0,1]^d} \phi_d(y) + \epsilon}{\inf_{y\in[0,1]^d} \phi_d(y) + 2\epsilon} < \frac{d}{d-1}
    \end{equation*}
    for all $x\in D_\epsilon$. For
    \begin{equation*}
        \alpha < \frac{1}{d-1}\l(1-\frac{\inf_{y\in[0,1]^d} \phi(y) + \epsilon}{\inf_{y\in[0,1]^d} \phi(y) + 2\epsilon}\r)
    \end{equation*}
    there exists $\rho'>0$ such that
    \begin{equation*}
        \kappa_d (c\log \rho)^{\frac{d}{d-1}} \rho^{d\alpha - \frac{d}{d-1}} < \int_{D_\epsilon} \rho^{-c\frac{2\kappa_{d-1}}{d\kappa_d} \phi_d(x)}\l(1+\cO(\rho r_\rho^{d+1})\r)\dd x
    \end{equation*}
    for all $\rho>\rho'$. Hence,
    \begin{equation*}
        \frac{\int_{D_\epsilon} \int_{B_d\l(x,\rho^\alpha r\r)}\P\l(x\notin \cC_{\rho}(r_\rho)\r) \dd y \dd x}{(\E V)^2} \to 0
    \end{equation*}
    as $\rho\to\infty$.
    
    Having found a range of values for $\alpha$ that makes the first summand vanish, we turn our attention to the second summand of \eqref{eq:Y2 integral division}:
    \begin{equation*}
        \int_{D_\epsilon}\int_{D_\epsilon \setminus B_d\l(x,\rho^\alpha r\r)}\P\l(x\notin \cC_{\rho}(r_\rho)\r)\P\l(y\notin \cC_{\rho}(r_\rho)| x\notin \cC_{\rho}(r_\rho)\r) \dd y \dd x.
    \end{equation*}    
    Let $x,y\in D_\epsilon$ be such that $|x-y| > \rho^{\alpha}r$. The number of cylinders covering $y$, given that $x$ is not covered, follows a Poisson distribution. To calculate the intensity, we first determine the expected number of cylinders covering both $x$ and $y$. By an argument similar to the proofs of \cref{lem:line crossing probability} and \cref{lem:line crossing intensity} (see also \cite[Lemma 3.1]{broman2019random}), there exists some constant $k>0$ such that for large enough $\rho$ and all $x,y\in D_\epsilon$ with $|x-y| > \rho^{\alpha}r$
    \begin{multline*}
        \E\sum_{i=1}^{\xi_\rho} \indi{X^{(i,\rho)}\cap B_d(x,r)\neq \emptyset}\indi{X^{(i,\rho)}\cap B_d(y,r)\neq \emptyset}\\ \leq  \l(\rho r^{d-1}\frac{2\kappa_{d-1}}{d\kappa_d} \phi_d(x) + \cO(\rho r_\rho^{d+1})\r) 
         \l(k \l(\frac{r}{\rho^{\alpha}r}\r)^{d-1}\frac{2\kappa_{d-1}}{d\kappa_d} \r) \eqqcolon A.
    \end{multline*}
    The first factor comes from the expected number of cylinders through $x$. The second factor is an upper bound on the probability that a line through $x$ also crosses $y$. Using the probability of a Poisson distributed random variable being zero, we obtain the following lower and upper bounds:
    \begin{align*}
        \P\l(y\notin \cC_{\rho}(r_\rho)\r) \leq \P\l(y\notin \cC_{\rho}(r_\rho)| x\notin \cC_{\rho}(r_\rho)\r)
        &\leq \P\l(y\notin \cC_{\rho}(r_\rho)\r) e^{A}.
    \end{align*}
    The term $A$ is of the order
    \begin{equation*}
        A = \cO\l((\log \rho ) \rho^{-\alpha(d-1)}\r),
    \end{equation*}
    so $A$ goes to zero as $\rho$ tends to infinity, meaning that
    \begin{equation*}
        \lim_{\rho\to\infty} \frac{\P\l(y\notin \cC_{\rho}(r_\rho)| x\notin \cC_{\rho}(r_\rho)\r)}{\P\l(y\notin \cC_{\rho}(r_\rho)\r)} = 1.
    \end{equation*}
    The coverage of $x$ and of $y$ are therefore sufficiently uncorrelated and the right-hand side of \eqref{eq:Y2 integral division} converges to $(\E V)^2$. It follows that
    \begin{equation*}
        \lim_{\rho\to\infty} \P\l(V>0\r) = 1,
    \end{equation*}
    which completes the proof of the lower bound.
    
    \paragraph{Upper bound} To derive the upper bound, we cover the unit cube with balls of radius $(1-\epsilon)r_\rho$. If any of these balls is not covered by the cylinders, the unit cube is not completely covered. The constant $c^*$ can then be derived using upper bounds on $\P\l([0,1]^d\not\subset \cC_{\rho}(r_\rho)\r)$. 
    
    Let $\epsilon<1/2$ be positive and consider $y\in [0,1]^{d}$. Then, by \cref{lem:line crossing intensity},
    \begin{equation*}
        \E\sum_{i=1}^{\xi_\rho} \indi{X^{(i,\rho)}\cap B_d(y,(1-\epsilon)r)\neq \emptyset} \geq \rho(1-\epsilon)^{d-1}r^{d-1}\frac{2\kappa_{d-1}}{d\kappa_d} \inf_{x\in[0,1]^d}\phi_d(x).
    \end{equation*}
    The probability that the $\epsilon$-ball around $y$ is not covered by the cylinder set is then bounded from above as follows:
    \begin{equation*}
        \P\l(B_d(y,\epsilon r) \not\subset \cC_{\rho}(r_\rho)\r) \leq \exp\l(-\rho(1-\epsilon)^{d-1}r^{d-1}\frac{2\kappa_{d-1}}{d\kappa_d} \inf_{x\in[0,1]^d}\phi_d(x)\r)
    \end{equation*}
    Let $\cI_{\rho,\epsilon}$ be a smallest set $\{y_i\}$ with $y_i\in[0,1]^d$ such that $\cup_i B(y_i,\epsilon r)$ contains $[0,1]^d$. The size of $\cI_{\rho,\epsilon}$ is of order $(\epsilon r)^{-d}$. Then, the union bound and the size of $\cI_{\rho,\epsilon}$ imply that there must exist some constant $k>0$ such that 
    \begin{align*}
        \P\l([0,1]^d\not\subset \cC_{\rho}(r_\rho)\r) &\leq \P\l(\bigcup_{y\in I_{\rho,\epsilon}}\{B(y,\epsilon r)\not\subset \cC_{\rho}(r_\rho)\}\r)\\
        &\leq k(\epsilon r)^{-d} \exp\l(-\rho(1-\epsilon)^{d-1}r^{d-1}\frac{2\kappa_{d-1}}{d\kappa_d} \inf_{x\in[0,1]^d}\phi_d(x)\r).
    \end{align*}
    Plugging in the value for $r$ and simplifying, the expression above becomes
    \begin{multline*}
        \P\l([0,1]^d\not\subset \cC_{\rho}(r_\rho)\r) \leq k\epsilon^{-d}(\log \rho)^{-d/(d-1)} \\
        \times \exp\l(\log \rho\l(\frac{d}{d-1} - c (1-\epsilon)^{d-1}\frac{2\kappa_{d-1}}{d\kappa_d} \inf_{x\in[0,1]^d}\phi_d(x)\r)\r).
    \end{multline*}
    This expression goes to zero whenever
    \begin{equation*}
        c > \frac{d^2 \kappa_d}{2(d-1)\kappa_{d-1}} \cdot \frac{1}{\inf_{x\in[0,1]^d}\phi_d(x)} \cdot \frac{1}{(1-\epsilon)^{d-1}}>c^*.
    \end{equation*}
    Since this must be true for all $\epsilon>0$, it follows that
    \begin{equation*}
        \lim_{\rho\to\infty }\P\l([0,1]^d\not\subset \cC_{\rho}(r_\rho)\r) =0
    \end{equation*}
    whenever $c > \frac{d^2 \kappa_d}{2(d-1)\kappa_{d-1}} \cdot\frac{1}{\inf_{x\in[0,1]^d} \phi_d(x)}=c^*$, which implies
    \begin{equation*}
        \lim_{\rho\to\infty }\P\l([0,1]^d \subset \cC_{\rho}(r_\rho)\r) =1.
    \end{equation*}
    This concludes to proof of the upper bound. The convergence from \cref{thm:coverage uniform} follows.
\end{proof}

\section{Generalizations}
\label{sec:generalization}
To prove \cref{thm:coverage}, we first show a result for the lower bound on $R_\rho(K)$ that applies to generalized cylinder sets based on stochastic processes satisfying some natural conditions.

Note that the distribution of
\begin{equation*}
    \bigcup_{i=1}^\infty X_t^{(i,\rho)} \oplus B'_{d-1}(o,r)
\end{equation*}
is constant in time by \cite[Proposition 1.3 (i)]{vandenberg1997}, meaning it is a Boolean model in $\R^{d-1}$ for all $t$, $r$ and $\rho$. This can be combined with known results on the coverage radius of the Boolean model to show the following proposition.
\begin{prop}[Lower bound]
    Let $X_t^{(i,\rho)}$ be constructed using some continuous-time stochastic process.
    Let $r_\rho = c\l(\frac{\log\rho}{\rho}\r)^{1/(d-1)}$ and suppose that $K$ satisfies for all $c>0$
    \begin{equation*}
        \cX_\rho(r_\rho K)\cap \big([r_\rho,1-r_\rho]^{d-1}\times\{t\}\big) \subset \bigcup_{i=1}^\infty X_t^{(i,\rho)} \oplus B'_{d-1}(o,a r_\rho)
    \end{equation*}
    almost surely for some $a>0$ and $t\in[0,1]$, and $\rho$ large enough. Then there exists a constant $c_1$ such that
    \begin{equation*}
        \P\l(\frac{R_\rho(K)}{\l(\log{\rho}/\rho\r)^{1/(d-1)}}\geq c_1\r) \to 1
    \end{equation*}
    as $\rho$ tends to infinity.
    \label{prop:lower bound}
\end{prop}
\begin{proof}
    Remember that for all $t'\in\R_+$ the set $\{X_{t'}^{(1,\rho)}, X_{t'}^{(2,\rho)}, X_{t'}^{(3,\rho)},\dots\}$ forms the support of a homogeneous Poisson point process. Then, for any time $t\in [0,1]$, the random variable
    \begin{equation*}
        R'_{\rho,t} = \inf\l\{r' \colon \cup_{i=1}^\infty X_t^{(i,\rho)} \oplus B'_{d-1}(o,r')\supset [r_\rho,1-r_\rho]^{d-1}\times\{t\}\r\}
    \end{equation*}
    is the coverage radius for the Boolean model. Since $\cX_\rho(r_\rho K)\cap \big([r_\rho,1-r_\rho]^{d-1}\times\{t\}\big) \subset \cup_{i=1}^\infty X_t^{(i,\rho)} \oplus B'_{d-1}(o,a r_\rho)$, it follows that $R_\rho(\frac{1}{a}K) \geq R'_{\rho,t}$. Note that the coverage radius scales as follows:
    \begin{equation*}
        R_\rho(\frac{1}{a}K) = a R_\rho(K).
    \end{equation*}
     
    Hence,
    \begin{align*}
        \P\l(\frac{R_\rho(K)}{\l(\log{\rho}/\rho\r)^{1/(d-1)}}\geq c_1\r) &= \P\l(\frac{R_\rho(\frac{1}{a}K)}{\l(\log{\rho}/\rho\r)^{1/(d-1)}}\geq ac_1\r)\\
        &\geq \P\l(\frac{R'_{\rho,t}}{\l(\log{\rho}/\rho\r)^{1/(d-1)}}\geq ac_1\r).
    \end{align*}
    By \cite[Theorem 1 Eq. (2c)]{calka_extreme_2014}, a constant $c_1>0$ exists such that the probability on the second line converges to one. It follows that for the same $c_1$
    \begin{equation*}
        \P\l(\frac{R_\rho(K)}{\l(\log{\rho}/\rho\r)^{1/(d-1)}}\geq c_1\r) \to 1
    \end{equation*}
    as $\rho$ tends to infinity.
\end{proof}

\subsection{Cylinders generated by lines (\cref{thm:coverage})}
\cref{thm:coverage} is is about cylinders generated by the Minkowski sums between a set of lines and compact sets $K$. The general statement for all $K$ containing $B'_{d-1}(o,r)$ follows from the statement for $B'_{d-1}(o,1)$ and $B_d(o,1)$.
\begin{prop}
    \cref{thm:coverage} holds for $K= B'_{d-1}(o,1)$ and $K= B_d(o,1)$.
    \label{prop:coverage}
\end{prop}
We first show how this Proposition can be used to prove \cref{thm:coverage}.
\begin{proof}[Proof of \cref{thm:coverage}]
    Since $K$ is a compact set containing $B'_{d-1}(o,r)$ for some $r>0$, we can choose $r_1,r_2>0$ such that
    \begin{align*}
        B'_{d-1}(o,r_1) \subseteq K \subseteq B_d(o,r_2)
    \end{align*}
    from which
    \begin{equation}
        R_\rho\l(B_d(o,r_2)\r) \leq R_\rho\l(K\r) \leq R_\rho\l(B'_{d-1}(o,r_1)\r).
        \label{eq:radius bounds}
    \end{equation}
    follows. By \cref{prop:coverage}, there exist $c_1$ and $c_2$ such that
    \begin{align*}
        \P\l(\frac{R_\rho\l(B_d(o,r_2)\r)}{\l(\log{\rho}/\rho\r)^{1/(d-1)}}\geq c_1\r) &\to 1\\
        \P\l(\frac{R_\rho\l(B'_{d-1}(o,r_1)\r)}{\l(\log{\rho}/\rho\r)^{1/(d-1)}}\leq c_2\r) &\to 1.
    \end{align*}
    \cref{thm:coverage} then follows from combining these probabilities with \eqref{eq:radius bounds}.
\end{proof}

\begin{proof}[Proof of \cref{prop:coverage}]
    In what follows, we consider $r_\rho$ of the form
    \begin{equation*}
        r_\rho = c\l(\frac{\log{\rho}}{\rho}\r)^{1/(d-1)}
    \end{equation*}
    for some constant $c>0$. Our goal is to show that there exist $c_1$ and $c_2$ such that
    \begin{align}
        \lim_{t\to\infty}\P\l(R_\rho < r_\rho\r) = 0 \qquad &\text{if } c < c_1
        \label{eq:lower bound}
        \\
        \lim_{t\to\infty}\P\l(R_\rho > r_\rho\r) = 0 \qquad &\text{if } c > c_2.
        \label{eq:upper bound}
    \end{align}
    
    Note that
    \begin{equation*}
        B'_{d-1}(o,1) \subset B_d(o,1),
    \end{equation*}
    which implies
    \begin{equation*}
        R_\rho\l(B'_{d-1}(o,1)\r) \geq R_\rho\l(B_d(o,1)\r).
    \end{equation*}
    Therefore, it is enough to derive the existence of the lower bound $c_1$ for $R_\rho\l(B_d(o,1)\r)$ and the upper bound for $R_\rho\l(B'_{d-1}(o,1)\r)$.

    \paragraph{Lower bound (Eq. \eqref{eq:lower bound})} For convenience of notation, let $R_\rho\coloneqq R_\rho\l(B_d(o,1)\r)$. Then,
    \begin{align*}
        \P\l(R_\rho < r_\rho\r) &\leq \P\l([r_\rho,1-r_\rho]^d\subset \cC_{\rho}(r_\rho)\r)\\
        & \leq \P\l([r_\rho,1-r_\rho]^{d-1}\times \{1\} \subset \cC_{\rho}(r_\rho)\r).
    \end{align*}
    We will therefore focus on covering the `upper' side of the unit cube.
    
    Suppose that $\ell=\{x+\alpha s\colon\, \alpha\in\R\}$ with $s\in\S^{d-1}$ is a line through $x\in[0,1]^{d-1}\times \{0\}$ that intersects some $y\in[r_\rho,1-r_\rho]^{d-1}\times \{1\}$. The $d$-th entry of $s$ must be at least $1/\sqrt{d}$ since the distance between two opposite corners of $[0,1]^d$ is $\sqrt{d}$. The probability of $s_d>1/\sqrt{d}$ is positive by \cref{cond:distribution S 2}. This bound on the angle gives us an upper bound on the volume of $[0,1]^{d-1}\times \{1\}$ covered by the cylinder:
    \begin{equation*}
        \Big(\ell\oplus r_\rho B_d(o,1)\Big)\cap \l([0,1]^{d-1}\times \{1\}\r) \subset B_{d-1}\l(y,\sqrt{d}r_\rho\r).
    \end{equation*}
    That is, the subset of $[0,1]^{d-1}\times\{1\}$ covered by the cylinder induced by $\ell$ is contained in a ball of radius $\sqrt{d} r_\rho$.

    Since this is true for any of the random cylinders in the cylinder set, we see that the condition in \cref{prop:lower bound} is satisfied. The lower bound \eqref{eq:lower bound} follows.
    
    \paragraph{Upper bound (Eq. \eqref{eq:upper bound})} 
    For the remainder of the proof, we let $R_\rho \coloneqq R_\rho(B'_{d-1}(o,1))$.
    Let us cover $[0,1]^d$ by $k_\rho=\ceil{r_\rho^{-1}}^d$ closed $d$-dimensional hypercubes $\{Q_i^{\rho}\}_{i\in[k_\rho]}$ of side length $r_\rho$ such that
    \begin{equation*}
        \bigcup_{i=1}^{k_\rho} Q_i^{\rho} = [0,1]^d.
    \end{equation*}
    Then,
    \begin{align*}
        \P\l(R_\rho> r_\rho\r) &= \P\l([0,1]^d\not\subset \cC'_{\rho,r_\rho}\r)\\
        &\leq \P\l(\bigcup_{i=1}^{k_\rho}\l\{Q_i^{\rho}\not\subset \cC'_{\rho,r_\rho}\r\}\r)\\
        &\leq \sum_{i=1}^{k_\rho} \P\l(Q_i^{\rho}\not\subset \cC'_{\rho,r_\rho}\r).
    \end{align*}
    We therefore need to find an upper bound on the probability of a cube not being covered by a cylinder.
    
    For $a_1,\dots,a_{d-1},h\in[0,1-r_\rho]^d$, let
    \begin{equation*}
        Q^{h} \coloneqq [a_1,a_1+r_\rho]\times[a_2,a_2+r_\rho]\times\cdots\times[a_{d-1}, a_{d-1}+r_\rho]\times [h,h+r_\rho]
    \end{equation*}
    be one of the hypercubes of side length $r_\rho$. Note that for any line intersecting two opposite sides of the cube $Q^h$, the corresponding cylinder fully covers the cube. We will find upper bounds for $\P(Q^h\not\subset \cC'_{\rho,r_\rho})$ by thinning the Poisson process of lines. \cref{fig:coverage proof} provides a sketch of the situation described below in two dimensions.
    \begin{figure}
        \centering
        \includegraphics[width=.4\textwidth]{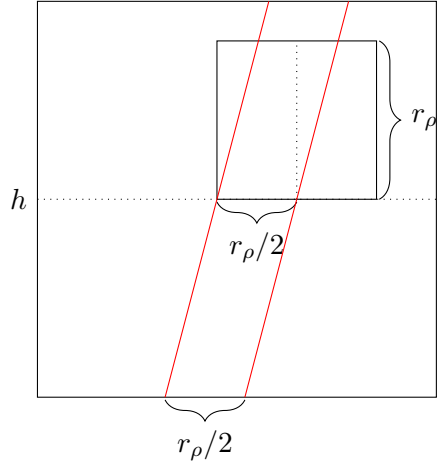}
        \caption{Illustration of the situation in the proof of \cref{prop:coverage}. A subcube is placed at height $h$. The right half of the subcube is completely contained in $[1/2,1]\times[0,1]$ Lines with a suitable direction intersecting the left half of this subcube exit through the top of the cube. When the direction of the lines is fixed, lines that intersect the left half of the small box must originate from an interval of length $r_\rho/2$ on the bottom side of $[0,1]^2$.}
        \label{fig:coverage proof}
    \end{figure}
    
    Let us divide $[a_1,a_1+r_\rho]\times\cdots\times[a_{d-1}, a_{d-1}+r_\rho]$ into $2^{d-1}$ subcubes with side length $r_\rho/2$ and $[0,1]^{d-1}$ into $2^{d-1}$ subcubes with side length $1/2$. Then one of the former subcubes will be completely contained in one of the subcubes with side length $1/2$. Without loss of generality, we may assume that
    \begin{equation*}
        [a_1+r_\rho/2,a_1+r_\rho]\times[a_2+r_\rho/2,a_2+r_\rho]\times\cdots\times[a_{d-1}+r_\rho/2, a_{d-1}+r_\rho] \subset [1/2,1]^{d-1}.
    \end{equation*}
    Consider $s\in\cS(1,1,\dots,1)$, where $\cS$ was defined in \eqref{eq:direction set}. Any line $L=\{x+\alpha s\colon \alpha\in\R\}$ that intersects $[a_1, a_1+r_\rho/2]\times \cdots \times [a_{d-1}, a_{d-1}+r_\rho/2]\times\{h\}$, also intersects $[a_1, a_1+r_\rho]\times \cdots \times [a_{d-1}, a_{d-1}+r_\rho]\times\{h+r_\rho\}$ and $[0,1]^{d-1}\times\{0\}$. That is, these lines intersect two opposite sides of the cube and also originate in $[0,1]^{d-1}\times\{0\}$.
    
    Based on this information, we remove every point $(x,s)$ from the process $\tilde\zeta$ when the corresponding line $\{x+\alpha s\colon \alpha\in\R\}$ does not intersect $[a_1, a_1+r_\rho/2]\times \cdots \times [a_{d-1}, a_{d-1}+r_\rho/2]\times\{h\}$ or when $s\notin \cS(1,1,\dots,1)$. All cylinders corresponding to the remaining lines cover $Q^h$. Let
    \begin{equation*}
        p_s \coloneqq \min \l\{\P\l(S_i\in \overline{\cS\l(z_1,\dots,z_{d-1}\r)}\r)\colon\, (z_1,\dots,z_{d-1})\in \{-1,1\}^{d-1}\r\},
    \end{equation*}
    which is positive since \cref{cond:distribution S 2} is satisfied. Using a similar argument as in the proof of \cref{lem:line crossing intensity} (also shown in \cref{fig:coverage proof}), one can show that the number of lines left is Poisson distributed with mean
    \begin{equation*}
        \lambda_{Q^h} \geq p_s \l(\frac{r_\rho}{2}\r)^{d-1} \rho.
    \end{equation*}
    Using this intensity bound, we obtain the following upper bound on the probability that no line intersects $Q^h$:
    \begin{align*}
        \P(Q^h\not\subset \cC'_{\rho,r_\rho}) &\leq \exp\l(-p_s \l(\frac{r_\rho}{2}\r)^{d-1} \rho\r)\\
        &= \exp\l(-2^{-(d-1)}p_s c^{d-1} \log{\rho}\r).
    \end{align*}
    Here, the factor $p_s$ comes from removing cylinders based on the direction and $(r_\rho/2)^{d-1}$ is the part of $Q^h$ the cylinders need to cross.
    
    The number of squares covering $[0,1]^d$ is bounded from above by $k_1 r_\rho^{-d}$ for some constant $k_1$. Summing over all squares results in the following upper bound:
    \begin{align*}
        \P\l(R_\rho\geq r_\rho\r) &\leq k_1 r_\rho^{-d}\exp\l(-2^{-(d-1)}p_s c^{d-1} \log{\rho}\r)\\
        &= k_1 \exp\l(-d\log{r_\rho}-2^{-(d-1)}p_s c^{d-1} \log{\rho}\r)\\
        &= k_2 \exp\l(d\l(\log{\rho}-\log{\log{\rho}}\r)-2^{-(d-1)}p_s c^{d-1} \log{\rho}\r),
    \end{align*}
    where $k_2>0$ is a constant. For $c^{d-1}>2^{d-1}d/p_s$, this upper bounds converges to zero as $\rho$ goes to infinity. Hence, for
    \begin{equation*}
        c > 2\l(\frac{d}{p_s}\r)^{1/(d-1)}
    \end{equation*}
    we have
    \begin{equation*}
        \P\l(\frac{R_\rho}{\l(\log{\rho}/\rho\r)^{1/(d-1)}}\leq c\r)\to 1 \qquad \text{as } \rho\to\infty,
    \end{equation*}
    which completes the proof for the upper bound.
\end{proof}

\subsection{Cylinders generated by Brownian motions (\cref{thm:brownian coverage})}
For the Brownian motion, we only consider the generalized cylinders created using the Minkowski sum with $K=B'(o,r)$. The proof can again be divided into a proof of a lower and an upper bound.

\begin{proof}[Proof of \cref{thm:brownian coverage}]
    We again consider $r_\rho$ of the form
    \begin{equation*}
        r_\rho = c\l(\frac{\log{\rho}}{\rho}\r)^{1/(d-1)}
    \end{equation*}
    for some constant $c>0$ and show limits of the form \eqref{eq:lower bound} and \eqref{eq:upper bound}. The lower bound \eqref{eq:lower bound} follows from an application of \cref{prop:lower bound} with $t=0$. In this proof, we let
    \begin{equation*}
        \cW_\rho(r) \coloneqq \cX_\rho(B'_{d-1}(o,1)).
    \end{equation*}
    denote the generalized cylinder set generated using the Brownian motions with $B'_{d-1}(o,r)$

    To prove the upper bound, we use the same idea as in the proof of \cref{prop:coverage}: we cover the unit cube by smaller sets and show an upper bound on the probability that at least one of them is not covered. We adapt the size of these boxes to apply the scaling property of Brownian motion.

    Let $Q_1,Q_2,\dots Q_{k_\rho}$ be $k_\rho$ translations of $Q = [0,\frac{1}{2\sqrt{d-1}}r_\rho]^{d-1}\times [0,r_\rho^2]$ covering the unit cube. Then there exists a constant $k>0$ such that $k_\rho\leq k r_\rho^{-(d-1)}r_\rho^{-2}$ for all $\rho$ large enough, so
    \begin{align*}
        \P\l([0,1]^d\not\subset \cW_\rho(r_\rho)\r) &\leq \sum_{i=1}^{k_\rho} \P\l(Q_i\not\subset \cW_\rho(r_\rho)\r)\\
        &\leq k r_\rho^{-(d-1)}r_\rho^{-2} \sup_{i\in[k_\rho]}\P\l(\nexists j \colon Q_i\subset W^{(j)} \oplus B'_{d-1}(o,r)\r),
    \end{align*}
    where the probability on the second line is the probability of the event that no
    single Brownian motion covers the entire cube $Q_i$. Similarly to the proof of \cref{prop:coverage}, we bound this probability by removing all stochastic processes $X_t^{(i,\rho)}$ from the process that do not satisfy two convenient conditions: they must enter $Q_i$ through its $(d-1)$-dimensional base and cover the cube by moving at most $\frac{1}{2\sqrt{d-1}}r$ into any direction for a duration of $r_\rho^2$. By the same argument as in the proof of \cref{lem:line crossing intensity}, the number of $X_t^{(i,\rho)}$ satisfying this condition is Poisson distributed.
    
    The probability that a one-dimensional Brownian motion $W_s$ with $W_0 = 0$ stays in the interval $[-\frac{r}{2\sqrt{d-1}}, \frac{r}{2\sqrt{d-1}}]$ for all $s\in[0,r^2]$ is
    \begin{align*}
        \P\l(W_s\in \l[-\frac{r}{2\sqrt{d-1}}, \frac{r}{2\sqrt{d-1}}\r]\, \forall s\in[0,r^2]\r) &= 1 - \P\l(\sup_{s\in [0,r^2]} |W_s| > \frac{r}{2\sqrt{d-1}}\r)\\
        &=1 - \P\l(\sup_{s\in [0,1]} |W_s| > \frac{1}{2\sqrt{d-1}}\r)\\
        &\eqqcolon p>0,
    \end{align*}
    where we used the scaling property of Brownian motions to get the second equality.
    For a $(d-1)$-dimensional Brownian motion with position described by $(d-1)$ independent Brownian motions in each component, the probability that the Brownian motion moves at most $\frac{1}{2\sqrt{d-1}}r$ into any direction is then $p^{d-1}$.
    
    The intensity of $X_t^{(i,\rho)}$ crossing a cube $Q_i$ is the lowest when $Q_i$ lies in an upper corner of $[0,1]^d$. If we define $Q_{\max}= [0,\frac{1}{2\sqrt{d-1}}r_\rho]^{d-1}\times [1-r_\rho^2,1]$, then 
    the expected number of processes hitting $[0,\frac{1}{2\sqrt{d-1}}r_\rho]^{d-1}$ at time $1-r_\rho^2$ is bounded from below by
    \begin{multline*}
        \E\l(\sum_{i=1}^{\xi_\rho} \ind\l(X_{1-r_\rho^2}^{(i,\rho)}\in [0,\frac{1}{2\sqrt{d-1}}r_\rho]^{d-1}\r)\r) \\
        \geq \l(\frac{1}{2\sqrt{d-1}}r_\rho\r)^{d-1}\l(\frac{1}{2\pi}\r)^{\frac{d-1}{2}} \exp\l(-\frac{1}{2}(d-1)^2\r)\rho,
    \end{multline*}
    where we used the density function of the normal distribution to bound the intensity of Brownian motions in an upper corner of the cube. Then    
    \begin{align*}
        \P\l(Q_i\not\subset \cW_\rho(r_\rho)\r) &\leq \P\l(Q_{\max}\not\subset \cW_\rho(r_\rho)\r)\\
        &\leq \exp\l(-\l(\frac{1}{2\sqrt{d-1}}r_\rho\r)^{d-1}\l(\frac{1}{2\pi}\r)^{\frac{d-1}{2}} \exp\l(-\frac{1}{2}(d-1)^2\r)\rho p^{d-1}\r).
    \end{align*}
    Plugging in the expression for $r_\rho$, this simplifies to
    \begin{align*}
        \P\l(Q_i\not\subset \cW_\rho(r)\r) &\leq \exp\l(-c^{d-1}\l(\frac{1}{2\sqrt{d-1}}\r)^{d-1}\log \rho\l(\frac{1}{2\pi}\r)^{\frac{d-1}{2}} \exp\l(-\frac{1}{2}(d-1)^2\r) p^{d-1}\r)\\
        &= \exp\l(-c_d \log \rho \r) =\rho^{-c_d},
    \end{align*}
    where $c_d \coloneqq c^{d-1} \l(\frac{1}{2\sqrt{d-1}}\r)^{d-1}\l(\frac{1}{2\pi}\r)^{\frac{d-1}{2}} \exp\l(-\frac{1}{2}(d-1)^2\r) p^{d-1}$.
    Hence,
    \begin{align*}
        \sum_{i=1}^{k_\rho} \P\l(Q_i\not\subset \cW_\rho(r_\rho)\r) &\leq k r_\rho^{-(d-1)} r_\rho^{-2} \rho^{-c_d}\\
        &= \frac{k}{c^{d+1}(\log \rho)^{1+\frac{2}{d-1}}}
        \rho^{-c_d + 1 + \frac{2}{d-1}}.
    \end{align*}
    For $c$ (and implicitly $c_d$) large enough, this expression tends to zero, meaning that $\P\l([0,1]^d\not\subset \cW_\rho(r_\rho)\r)$ goes to zero as $\rho$ tends to infinity. This completes the proof of the upper bound.
\end{proof}

\begin{rem}
    The proofs of the upper bounds in \cref{prop:coverage} and \cref{thm:brownian coverage} both rely on the idea of covering the unit cube by smaller cubes. This could naturally be extended to other stochastic processes: if there is a non-zero intensity of stochastic processes everywhere in $[0,1]^d$ and these processes tend to not move too fast, then the upper bound follows.

    A possible extension would be the Brownian motion with linear drift, where $X_t^{(i,\rho)} = X_0^{(i,\rho)} + W_t^{(i,\rho)} + tV$ for some random velocity vector $V\in \R^{d-1}$. This requires some adjustments to the proof above. Instead of requiring that a Brownian motion moves at most $\frac{r_\rho}{2}\sqrt{d-1}$ in a time interval of length $r_\rho^2$, we restrict ourselves to cylinders such that $W_t^{(i,\rho)}$ moves at most $\frac{r_\rho}{4}\sqrt{d-1}$ in an interval of length $r_\rho^2$ and $\|Vr_\rho^2\|\leq \frac{r_\rho}{4}\sqrt{d-1}$ too. This leads to similar bounds as in the proof above.
\end{rem}

\section*{Funding information}
Funded by the Deutsche Forschungsgemeinschaft (DFG, German Research Foundation) – Project-ID 531542011.

\addcontentsline{toc}{section}{References}
\bibliographystyle{abbrv}

\bibliography{ref}

\begin{thebibliography}{10}

\bibitem{aschenbruck2021}
N.~Aschenbruck, S.~Bussmann, and H.~D{\"o}ring.
\newblock Asymptotics of a time bounded cylinder model.
\newblock {\em Probability in the Engineering and Informational Sciences}, 2022+.

\bibitem{broman2019random}
E.~I. Broman and F.~Mussini.
\newblock Random cover times using the poisson cylinder process.
\newblock {\em Latin American Journal of Probability and Mathematical Statistics}, 16:1165--1199, 2019.

\bibitem{calka_extreme_2014}
P.~Calka and N.~Chenavier.
\newblock Extreme values for characteristic radii of a {Poisson}-{Voronoi} {Tessellation}.
\newblock {\em Extremes}, 17(3):359--385, Sept. 2014.

\bibitem{chenavier2016}
N.~Chenavier and R.~Hemsley.
\newblock Extremes for the inradius in the poisson line tessellation.
\newblock {\em Advances in Applied Probability}, 48(2):544--573, 2016.

\bibitem{dvoretzky1956}
A.~Dvoretzky.
\newblock On covering a circle by randomly placed arcs.
\newblock {\em Proceedings of the National Academy of Sciences of the United States of America}, 42(4):199--203, 1956.

\bibitem{gilbert1961}
E.~N. Gilbert.
\newblock Random plane networks.
\newblock {\em Journal of the Society for Industrial and Applied Mathematics}, 9(4):533--543, 1961.

\bibitem{goudsmit1945}
S.~Goudsmit.
\newblock Random distribution of lines in a plane.
\newblock {\em Reviews of Modern Physics}, 17(2-3):321, 1945.

\bibitem{janson_random_1986}
S.~Janson.
\newblock Random coverings in several dimensions.
\newblock {\em Acta Mathematica}, 156(none):83--118, Jan. 1986.

\bibitem{penrose2021}
M.~D. Penrose.
\newblock Random {Euclidean} coverage from within.
\newblock {\em Probability Theory and Related Fields}, jan 2023.

\bibitem{peres_mobile_2013}
Y.~Peres, A.~Sinclair, P.~Sousi, and A.~Stauffer.
\newblock Mobile geometric graphs: detection, coverage and percolation.
\newblock 156(1):273--305.

\bibitem{stevens_solution_1939}
W.~L. Stevens.
\newblock Solution to a geometrical problem in probability.
\newblock {\em Annals of Eugenics}, 9(4):315--320, Dec. 1939.

\bibitem{vandenberg1997}
J.~{van den Berg}, R.~Meester, and D.~G. White.
\newblock Dynamic boolean models.
\newblock {\em Stochastic Processes and their Applications}, 69(2):247--257, 1997.

\end{thebibliography}
\end{document}